\documentclass[psamsfonts]{amsart}

\usepackage{amssymb,amsfonts}
\usepackage[all,arc]{xy}
\usepackage{enumerate}
\usepackage{mathrsfs}
\usepackage{amsmath}
\usepackage{fullpage}
\usepackage{hyperref}
\usepackage{graphicx}

\newtheorem{theorem}{Theorem}

\newtheorem{claim}{Claim}

\newtheorem{definition}{Definition}

\newtheorem{proposition}{Proposition}

\numberwithin{equation}{section}

\title{K\"ahler metrics with cone singularities along a divisor of bounded Ricci curvature}

\author{Martin de Borbon}

\date{DECEMBER, 2016}

\begin{document}

\begin{abstract}

Let $D$ be a smooth divisor in a compact complex manifold $X$ and let $\beta \in (0,1)$. We use the liner theory developed by Donaldson \cite{donaldson1} to show that in any positive co-homology class on $X$ there is a K\"ahler metric  with cone angle $2\pi\beta$ along $D$ which has bounded Ricci curvature. We use this result together with the Aubin-Yau continuity method to give an alternative proof of a well-known existence theorem for Kahler-Einstein metrics with cone singularities.

\end{abstract}

\maketitle


\section{Introduction}
Let $X$ be a compact complex manifold, $D \subset X$ a smooth divisor and $\beta \in (0,1)$. Let $\omega$ be a $C^{\alpha}$ K\"ahler metric on $X$ with cone angle $2\pi\beta$ along $D$ -see Section \ref{BackgroundSection} for the definitions-; $\alpha \in (0,1)$ is the H\"older exponent, we also require that $\alpha \leq (1/\beta)-1$. Let $0< \alpha' < \alpha$ and $\epsilon >0$.

\begin{theorem} \label{Theorem1}
	 There is $\phi \in C^{2, \alpha'}(X)$ with $\| \phi \|_{2, \alpha'} < \epsilon$ such that the Ricci form of \( \omega_{\phi} = \omega + i \partial \overline{\partial} \phi \) extends to $X$ as a 2-form, smooth with respect to the complex coordinates. In particular $\omega_{\phi}$ has bounded Ricci curvature.
\end{theorem}

We take $\alpha' < \alpha$ in order to approximate as much as we like the relevant $C^{\alpha}$ function (`Ricci potential' of $\omega$) by smooth functions. Theorem \ref{Theorem1} follows by performing a suitable and small change in the volume form of $\omega$, this is done via the Implicit Function Theorem and uses the linear theory developed in \cite{donaldson1}.

In a different direction; it is conjectured that the existence of Kahler metrics with cone singularities of bounded sectional curvature imposes strong restrictions on the complex geometry of the pair $(X, D)$ -see \cite{Arezzocurvature}-. More precisely, if we denote the normal bundle of $D$ as $\nu_D$ and the tangent bundles as $TX$ and $TD$; then it is expected a holomorphic splitting \( TX|_D = TD \oplus \nu_D \). Theorem \ref{Theorem1} says that there are no such restrictions for the case of Ricci curvature.

As an application of Theorem \ref{Theorem1} we give an alternative proof of a well-established existence result for K\"ahler-Einstein metrics with cone singularities (KEcs). These are metrics with cone angle \(2\pi\beta\) along $D$ which satisfy \( \mbox{Ric}(g) = \lambda g \) in the complement of $D$ for some constant $\lambda$.

\begin{theorem} \label{Theorem2}
	
	\begin{enumerate}
		\item If \( c_1 (X) - (1-\beta)c_1(D) < 0 \); then there is a unique KEcs with $\lambda=-1$.
		\item If \( c_1(X)= (1-\beta) c_1(D) \); then there is a unique KEcs with $\lambda=0$ in each K\"ahler class.
	\end{enumerate}

\end{theorem} 

We prove Theorem \ref{Theorem2} by means of the classical Aubin-Yau continuity path, starting with a metric of bounded Ricci curvature. The openness follows from \cite{donaldson1} and the closedness from standard a priori estimates. The $C^0$ estimate uses the maximum principle when $\lambda=-1$ (see \cite{Jeffress}) and Moser iteration when $\lambda=0$ (see \cite{brendle}). The $C^2$ estimate follows from the maximum principle applied to the Chern-Lu inequality, together with the fact that there is a reference metric with bisectional curvature bounded above (see \cite{JMR}). Finally, the $C^{2, \alpha}$ estimate follows from the interior Schauder estimates in \cite{ChenWang}.

The proof of Theorem \ref{Theorem2} presents -in this simpler compact setting- the arguments in \cite{martin} used to establish an existence theorem for asymptotically conical Ricci-flat K\"ahler metrics with cone singularities. In that context, the method we use to show Theorem \ref{Theorem1} serves to produce asymptotically conical metrics which are Ricci-flat outside a compact set -see Proposition 6 in \cite{martin}-. 

\subsection*{Acknowledgments}  This article contains material from the author's PhD Thesis at Imperial College, founded by the European Research Council Grant 247331 and defended in December 2015. I wish to thank my supervisor, Simon Donaldson, for his encouragement and  support.

\section{Background} \label{BackgroundSection}

\subsection{Linear Theory}
Fix \( 0 < \beta < 1 \). We work on \( \mathbb{C}^n \) with complex coordinates \( z_1, \ldots, z_n \). Consider the model metric
\begin{equation} \label{FlatMetric}
	 g_{(\beta)} = \beta^2 |z_1|^{2\beta-2} |dz_1|^2 + \sum_{j=2}^n |dz_j|^2,
\end{equation}
which has cone singularities of total angle \(2\pi\beta\) along \(\{z_1=0\}\). It induces a distance $d_{\beta}$ and therefore, for each \( \alpha \in (0, 1) \), a H\"older  semi-norm
\begin{equation} \label{Holder seminorm}
 [u]_{\alpha} = \sup_{x, y} \frac{|u(x) - u(y)|}{d_{\beta}(x, y)^{\alpha}} 
\end{equation}
on continuous functions defined on domains of \( \mathbb{C}^n \). If we write \(z_1= r^{1/\beta} e^{i\theta} \), then 
\(g_{(\beta)} = dr^2 + \beta^2 r^2 d\theta^2 + \sum_{j=2}^{n} |dz_j|^2\). 
In these \emph{cone coordinates}, \( (re^{i\theta}, z_2, \ldots, z_n) \), \( g_{(\beta)} \) is quasi-isometric to the Euclidean metric -indeed \( \beta^2 g_{Euc} \leq g_{(\beta)} \leq g_{Euc} \)-; therefore \ref{Holder seminorm} becomes equivalent to the standard H\"older semi-norm with respect to the Euclidean distance. 

We want to define H\"older continuous 1-forms. Set $\epsilon = dr + i \beta r d\theta$. A $(1,0)$ form $\eta$ is called $C^{\alpha}$  if $\eta= u_1 \epsilon + \sum_{j=2}^n u_j dz_j$ with $u_1, u_j$ $C^{\alpha}$ functions in the usual sense in the cone coordinates; it is also required that $u_1  =0$  on the singular set $\lbrace z_1 = 0 \rbrace$. Note that if we change $\epsilon$ with $\tilde{\epsilon} = e^{i\theta} \epsilon = \beta |z_1|^{\beta-1} dz_1$, say, in the definition then the vanishing condition implies that we get the same space. We move on and consider a 2-form $\eta$ of type $(1, 1)$, we use the basis $\lbrace \epsilon \overline{\epsilon}, \epsilon d\overline{z_j}, dz_j\overline{\epsilon}, dz_jd\overline{z_k} \rbrace$ for \( j, k=2, \ldots, n \). We say that $\eta$ is $C^{\alpha}$ if its components are $C^{\alpha}$ functions; we also require the components  corresponding to $\epsilon d\overline{z_j},  dz_j \overline{\epsilon}$ to vanish on \( \{z_1 =0\} \).  
Finally, we set $C^{2, \alpha}$ to be the space of $C^{\alpha}$ (real) functions $u$ such that $\partial u, i\partial \overline{\partial} u$ are $C^{\alpha}$. It is straightforward to introduce norms; we define the $C^{\alpha}$ norm of a function $\|u\|_{\alpha}$ as the sum of its $C^0$ norm $\|u\|_0$ and its $C^{\alpha}$ semi-norm $[u]_{\alpha}$.
The $C^{2, \alpha}$ norm of a function $u$, denoted by \( \|u\|_{2, \alpha} \), is the sum of  $\|u\|_{\alpha}$,  the $C^{\alpha}$ norm of the components of $\partial u$ in  and the $C^{\alpha}$ norm of the components of $i \partial \overline{\partial} u$.

Let $X$ be a compact complex manifold, $D \subset X$ a smooth divisor and $g$ a smooth K\"ahler metric on the complement of $D$. 

\begin{definition} \label{DefinitionMetricConeSing}
	We say that $g$ is a \(C^{\alpha}\) metric with cone angle $2 \pi \beta$ along $D$, if for every $p \in D$ we can find complex coordinates $(z_1, \ldots, z_n)$ centred at $p$ in which
	\begin{itemize}
		\item  $D = \{z_1 =0\}$.
		\item \((1/C) g_{(\beta)} \leq g \leq C g_{(\beta)}\) for some \( C>0 \).
		\item There is a local K\"ahler potential for \(g\) which belongs to \(C^{2, \alpha}\).  
	\end{itemize}
\end{definition}

It is easy to show that if \(g\) satisfies Definition \ref{DefinitionMetricConeSing}, then its tangent cone at points of \(D\) is \(g_{(\beta)}\). It is also straightforward to see that its K\"ahler form \(\omega\) defines a closed current on \(X\) -with zero Lelong numbers at points of \(D\)- and therefore there is a positive de Rham co-homology class \( [\omega] \). Nevertheless, we won't make use of these facts.

Set $\lbrace v_1, \ldots, v_n \rbrace$ to be the  vectors 
\begin{equation} \label{VECTORS}
v_1 = |z_1|^{1-\beta} \frac{\partial}{\partial z_1}, \hspace{3mm} v_j = \frac{\partial}{\partial z_j} \hspace{2mm} \mbox{for} \hspace{1mm} j = 2, \ldots n . 
\end{equation}
Note that, with respect to $g_{(\beta)}$, these vectors are orthogonal  and their length is constant.
In the complement of $D$ we have smooth functions $g_{i\overline{j}} = g (v_i, \overline{v}_j)$
which admit a  H\"older continuous extension to $D$. The matrix $(g_{i\overline{j}} (p))$ is positive definite  and  $g_{1\overline{j}} =0$  when $j \geq 2$ and $z_1=0$.
It is interesting to note that Definition \ref{DefinitionMetricConeSing} is independent of the complex coordinates $z_1, \ldots, z_n$ only if we add the restriction that $\alpha \leq \beta^{-1} -1$. Indeed, assume for simplicity that $g$ is a metric in a domain of $\mathbb{C}^2$, with standard complex coordinates $(\tilde{z}_1, \tilde{z}_2)$,  of cone angle $2\pi\beta$ along $D =\{ \tilde{z}_1 =0 \}$. We get smooth functions $\tilde{g}_{i\overline{j}}$ on the complement of $D$ which extend H\"older continuously to $D$. Set $\tilde{z}_1= z_1$ and $\tilde{z}_2 = z_1 + z_2$, so that
$$ \frac{\partial}{\partial z_1} = \frac{\partial}{\partial \tilde{z}_1} + \frac{\partial}{\partial \tilde{z}_2} , \hspace{4mm} \frac{\partial}{\partial z_2} = \frac{\partial}{\partial \tilde{z}_2} .$$
In the coordinates $(z_1, z_2)$, 
$$ g_{1\overline{1}} = \tilde{g}_{1\overline{1}} + |z_1|^{1-\beta} (\tilde{g}_{1\overline{2}} + \tilde{g}_{2\overline{1}}) + |z_1|^{2-2\beta} \tilde{g}_{2\overline{2}}, \hspace{3mm} g_{1\overline{2}}= \tilde{g}_{1\overline{2}} + |z_1|^{1-\beta} \tilde{g}_{2\overline{2}}, \hspace{3mm} g_{2\overline{2}} = \tilde{g}_{2\overline{2}} . $$
However, the function $|z_1|^{1-\beta}$ belongs to $C^{\alpha}$ only if $\alpha \leq \beta^{-1} -1$. 

There are two types of coordinates we can consider around $D$. The first is given by holomorphic coordinates $z_1, \ldots, z_n$ in which $D= \lbrace z_1 =0 \rbrace$. In the second we replace $z_1$ with $r e^{i\theta}$, by means of \(z_1=r^{1/\beta}e^{i\theta}\), and leave $z_2, \ldots, z_n$ unchanged; we refer to the last as cone coordinates. In other words,  there are two relevant differential structures on $X$ in our situation. One is given by the complex manifold structure we started with, the other  is given by declaring the cone coordinates to be smooth. The two structures are clearly equivalent by a map modelled on
\[ (re^{i\theta}, z_2, \ldots, z_n) \to (r^{1/\beta}e^{i\theta}, z_2, \ldots, z_n) \]
in a neighborhood of $D$.  The notion of a function being H\"older continuous (without specifying the exponent) is independent of the coordinates we take, however the value of its exponent does depend. We set \( C^{\alpha}(X) \) to be the space of H\"older continuous functions of exponent \( \alpha\) with respect to the cone coordinates -this agrees with the space of \( C^{\alpha \beta} \) functions with respect to the complex coordinates-. Taking a finite covering of $X$ by complex coordinate charts, it is straightforward to define the space \( C^{2, \alpha}(X) \) and endow it with a norm which makes it into a Banach space. The main result we want to recall is the following
\begin{theorem} \label{LinearTheorem}
	Assume that \( \alpha < (1/\beta)-1 \); then \( \Delta_{g} : C^{2, \alpha}(X) \to C^{\alpha}(X) \) is a Fredholm operator of index $0$.
\end{theorem}

Theorem \ref{LinearTheorem} is proved in \cite{donaldson1}. In this article \( \Delta_g \) denotes the negative or `analyst' Laplacian of \(g\).

\subsection{Standard reference metric}
Let \( \Lambda \) be the complex line bundle associated to \(D\), \(h\) a smooth Hermitian metric on it and $s$ a holomorphic section of \(\Lambda\) with \(s^{-1}(0)=D \). Let \( \Omega \) be a smooth K\"ahler metric on \(X\); for \(\delta>0\) set

\begin{equation}
	\omega = \Omega + \delta i \partial \overline{\partial} |s|_h^{2\beta} .
\end{equation}
We have the following
\begin{proposition} \label{ReferenceMetric}
If \(\delta \) is sufficiently small, then \( \omega \) is \(C^{\alpha}\) according to Definition \ref{DefinitionMetricConeSing} with \( \alpha = (1/\beta)-1 \). Moreover, its sectional curvature is uniformly bounded above.
\end{proposition}

\begin{proof}
	By compactness, it is enough to work locally. Let $F$ be a smooth positive function and let $\Omega$ be a smooth K\"ahler form, both defined on a domain in $\mathbb{C}^n$ which contains the origin. Consider the $(1,1)$ form
	\begin{equation} \label{LOC MET}
	\omega = \Omega +   i \partial \overline{\partial} (F |z_1|^{2\beta}) .
	\end{equation}
	Straightforward calculation gives us that
	$$  i \partial \overline{\partial} (F |z_1|^{2\beta}) = |z_1|^{2\beta} i \partial \overline{\partial} F + \beta |z_1|^{2\beta-2} \left( \overline{z}_1 i dz_1 \wedge \overline{\partial}F + z_1 i\partial F \wedge d \overline{z}_1 \right) + \beta^2 |z_1|^{2\beta-2} F idz_1 \wedge d\overline{z}_1 . $$
	Let $I$ be the complex structure of $\mathbb{C}^n$ and $ g = \omega (., I.)$. Let $v_1, \ldots, v_n$ be as in \ref{VECTORS}. We want to compute $g_{i\overline{j}} = g(v_i, \overline{v}_j)$. Write $ \Omega = \sum_{i, j =1}^n \Omega_{i\overline{j}} i dz_i \wedge  d\overline{z}_j$. Note that the coefficients $\Omega_{i\overline{j}}$ are given by the contraction of $\Omega$ with the standard coordinate vectors $\partial/ \partial z_i$, $\partial /\partial \overline{z}_j$, while to obtain $g_{i\overline{j}}$ we must contract $g$ with $v_i, \overline{v}_j$. It is easy to check that
	
	$$ g_{1\overline{1}} = |z_1|^{2-2\beta} \Omega_{1\overline{1}} + |z_1|^2 \frac{\partial^2 F}{\partial z_1 \partial \overline{z}_1} + \beta \left( z_1 \frac{\partial F}{\partial z_1} + \overline{z}_1 \frac{\partial F}{\partial \overline{z}_1} \right) + \beta^2 F  ; $$
	
	$$g_{1\overline{j}} = |z_1|^{1-\beta} \Omega_{1\overline{j}} + |z_1|^{1+\beta} \frac{\partial^2 F}{\partial z_1 \partial \overline{z}_j } + \beta |z_1|^{\beta -1} \left( z_1 \frac{\partial F}{\partial z_j} + \overline{z}_1 \frac{\partial F}{\partial \overline{z}_j} \right) \hspace{4mm} \mbox{for} j\geq 2 ; $$
	
	$$ g_{j\overline{k}}= \Omega_{j\overline{k}} + |z_1|^{2\beta} \frac{\partial^2 F}{\partial z_j \partial \overline{z}_k} \hspace{4mm} \mbox{for} j, k \geq 2 . $$
	It is then clear that if \(|z_1|\) is sufficiently small, then $g$ defines a $C^{\alpha}$  K\"ahler metric with $\alpha = \beta^{-1}-1$. 
	
	There is a useful way of thinking of $g$, due to J. Sturm -see \cite{Rubinstein}, Lemma 3.14- : On $\mathbb{C}^{n+1}$ with standard complex coordinates $(z_1, \ldots, z_{n+1})$ consider the $(1, 1)$ form
	$$ \Gamma = \Omega + i\partial \overline{\partial} (F |z_{n+1}|^2) . $$
	This form defines a smooth K\"ahler metric on $\mathbb{C}^{n+1}$ in a neighbourhood of $0$. Let us delete a ray in the complex plane corresponding to the $z_1$ variable and define
	$$ \Phi (z_1, \ldots, z_n) = (z_1,  \ldots, z_n, z_1^{\beta}) , $$
	so that $\omega = \Phi^{*} \Gamma$. The pull-back of $\Gamma$ by $\Phi$ is independent of the branch of $z_1^{\beta}$ that we take and we can think of the metric $g$ in the complement of $D$ as the restriction of the smooth metric defined by $\Gamma$ to a smooth complex hyper-surface in $\mathbb{C}^{n+1}$. A well-known principle says that the holomorphic sectional curvature of a complex submanifold of a K\"ahler manifold is less or equal than that of the ambient manifold, see Section 0.5 in Griffiths-Harris \cite{GriffithsHarris}. We conclude that we can restrict $g$ to a smaller neighbourhood of $0$ if necessary so that its sectional curvature is uniformly bounded above. 
	
\end{proof}

It is easy to check that \( [\omega]= [\Omega] \) as de Rham co-homology classes. We refer to \( \omega \) as the `standard reference metric'. It follows from the computations in the appendix of \cite{JMR} that the sectional curvature of \( \omega \) is unbounded below at \(D\); it then follows that the same holds for its Ricci curvature. We remark that this negative curvature phenomena is not inherent to the cone singularities, it is a consequence of the particular definition of \(\omega\). A good example to have in mind is the following: Let $a$ be a real number with $|a| <1$. Consider the metric defined in the unit disc of the complex numbers, 
$$ g_a = (a + |z_1|^{2\beta-2} )|dz_1|^2 . $$
Its Gaussian curvature is given by
$$ K_a = - 4(\beta-1)^2 a \frac{|z_1|^{2-4\beta}}{(1 + |z_1|^{2-2\beta}a)^3} . $$
If $1/2 < \beta <1$, then $K_a$ is unbounded below when $a>0$ and unbounded above if $a<0$. In higher dimensions we can take the product of $g_a$ with a flat euclidean factor $\mathbb{C}^{n-1}$.

For reference in the future; we recall that on a K\"ahler manifold there are the notions of sectional curvature, holomorphic sectional curvature and bisectional  curvature.  A uniform (upper or lower) bound in any of these three implies a uniform (upper or lower) bound on the other two.

\section{Proof of Theorem \ref{Theorem1}}

Consider the functional \( H \) given by 
\begin{equation}
	H(\phi)= \log \frac{\omega_{\phi}^n}{\omega^n}, 
\end{equation}
where \( \omega_{\phi}= \omega + i \partial\overline{\partial} \phi \); it is defined in a suitable neighbourhood of \(0\) in \( C^{2, \alpha'}(X) \) and takes values in  $C^{\alpha'}(X)$. Let \( v = \int_X \omega^n \) and \( M= \{ h \in C^{\alpha'}(X) \hspace{1mm} \mbox{s.t.} \hspace{1mm} \int_X e^h \omega^n = v \} \); integration by parts shows that \( \int_X \omega_{\phi}^n = \int_X \omega^n \) and therefore \(H(\phi)\in M \) for any \(\phi\). Clearly \(H(0)=0 \); standard computations show that \(H\) is \(C^1\) and that its derivative at \(0\) agrees with \( \Delta_g \). Write \( T_0 M = \{ \psi \in C^{\alpha'}(X) \hspace{1mm} \mbox{s.t.} \hspace{1mm} \int_X \psi \omega^n = 0 \}\)  for the tangent space of \(M\) at \(0\); and let \( L = \{ \phi \in C^{2, \alpha'}(X) \hspace{1mm} \mbox{s.t.} \hspace{1mm} \int_X \phi \omega^n = 0 \}  \). It follows from Theorem \ref{LinearTheorem} and the Implicit Function Theorem that \(H\) defines a diffeomorphism between small neighborhoods of \(0\), say \( U \subset L \) and \( V \subset M \). We can assume that \(U \subset B(0, \epsilon) \), the ball centred at the origin of radius \(\epsilon\); and that \(B(0, 2\mu) \cap M \subset V \) for some \( \mu>0 \).

On the other hand, a standard formula in K\"ahler geometry tells us that, in the complement of \(D\)
\begin{equation} \label{Ricci1}
	\mbox{Ric}(\omega_{\phi})- \mbox{Ric}(\omega) = -i \partial \overline{\partial} H(\phi);
\end{equation}
and 
\begin{equation} \label{Ricci2}
	\mbox{Ric}(\omega)- \mbox{Ric}(\Omega) = i \partial \overline{\partial} \log \frac{\Omega^n}{\omega^n} =  i \partial \overline{\partial} \log \frac{ |s|_h^{2\beta-2}\Omega^n}{\omega^n} + (1-\beta) i \partial \overline{\partial} \log |s|_h^2 .  
\end{equation}
Since \( \omega \) is a \(C^{\alpha}\) metric, the function

\[ F= \log \frac{ |s|_h^{2\beta-2}\Omega^n}{\omega^n} \]
belongs to \(C^{\alpha}(X)\). Since \(\alpha' < \alpha \), there is \( \tilde{h} \in B(0, \mu) \subset C^{\alpha'}(X) \) such that \( F - \tilde{h} \) is a smooth function on \(X\) with respect to the complex coordinates. Note that \( e^{-\mu} v \leq \int_X e^{\tilde{h}}\omega^n \leq e^{\mu} v \), so we can add a constant to \(\tilde{h}\) to get \(h \in V \) such that \(F-h\) is smooth. Write \(h=H(\phi)\) with \(\phi \in U \); \ref{Ricci1} together with \ref{Ricci2} give us

\begin{equation} \label{Ricci3}
	\mbox{Ric}(\omega_{\phi}) = \mbox{Ric}(\Omega) + (1-\beta) i \partial \overline{\partial} \log |s|_h^2 + i \partial \overline{\partial} \left( F-H(\phi) \right) .
\end{equation}
Note that $i \partial \overline{\partial} \log |s|_h^2$ extends as a smooth 2-form on \(X\), indeed it is the standard representative for \( -2\pi c_1 (\Lambda) \). Theorem \ref{Theorem1} then follows from \ref{Ricci3}.

For the sake of clarity we remark that in the proof we use standard derivatives in the complement of \(D\). If we were using currents and working globally on \(X\); then we would have to include the term \( 2 \pi (1- \beta) [D] \) on the right hand sides of equations \ref{Ricci2} and \ref{Ricci3}, with \([D]\) the current of integration along \(D\).

\section{Proof of Theorem \ref{Theorem2}}

We prove the case of negative Ricci, the case of Ricci-flat metrics goes along the same lines; the major difference is in the \(C^0\) estimate, in which Moser iteration is used instead of the maximum principle. There are no difficulties in extending the Moser iteration technique to the setting of metrics with cone singularities, for the details we refer to \cite{brendle}. We concentrate on the existence part; the uniqueness follows from the maximum principle -see \cite{Jeffress}-.

The hypothesis that $c_1 (X) - (1 - \beta) c_1 ([D]) < 0$ implies that there is a smooth K\"ahler form $\Omega$ such that
$ - (2\pi)^{-1} [ \Omega ] = c_1 (X) - (1 - \beta) c_1 ([D])$.
Take $s$ to be a holomorphic section of $\Lambda$ such that $s^{-1}(0) = D$ and let $h$ be a smooth Hermitian metric on $\Lambda$. Fix $\delta >0$ so that we have the reference metric $ \omega = \Omega + \delta i \partial \overline{\partial} |s|_h^{2\beta}$ 
of Proposition \ref{ReferenceMetric}.
\begin{claim}
	There is a $C^{\alpha}$ function $f$ on $X$, smooth in the complement of $D$,  such that
	$ \mbox{Ric}(\omega) = - \omega + i \partial \overline{\partial} f.$ We refer to \(f\) as the Ricci potential of $\omega$.
\end{claim}

\begin{proof}
	The co-homology condition on $\Omega$ implies that there is a smooth function $F$ on $X$ with
	$ i\partial \overline{\partial} F = \Omega + \mbox{Ric}(\Omega) + (1- \beta) i\partial \overline{\partial} \log |s|_h^2 $.
	We use that
	$ \mbox{Ric}(\omega) - \mbox{Ric}(\Omega) = i \partial \overline{\partial} \log \left( \Omega^n/\omega^n \right)$ 
	to obtain
	$$ \mbox{Ric}(\omega) = \mbox{Ric}(\Omega) +  i \partial \overline{\partial} \log \left( \frac{\Omega^n}{\omega^n} \right) =  i\partial \overline{\partial} F - \Omega -  i\partial \overline{\partial} \log  \left( \frac{|s|_h^{2-2\beta} \omega^n}{\Omega^n} \right) = -\omega + i\partial \overline{\partial}f  ; $$
	where
	$$ f = F + \delta |s|_h^{2\beta} -  \log  \left( \frac{|s|_h^{2-2\beta} \omega^n}{\Omega^n} \right) . $$
	Since \(\omega\) is \(C^{\alpha}\), we see that $f$ is a smooth function in the complement of $D$ which extends as a $C^{\alpha}$ function to $X$ with \( \alpha = (1/\beta)-1 \).
\end{proof}

We want to find $u \in C^{2, \alpha}$ a solution of

\begin{equation} \label{NEG KE}
(\omega + i \partial \overline{\partial} u )^n = e^{ f + u} \omega^n .
\end{equation}
It is  easy to argue that if we set $\omega_{KE} = \omega + i \partial \overline{\partial} u$, then $\omega_{KE}$ defines a K\"ahler metric with cone angle $2\pi \beta$ along $D$ and $\mbox{Ric}(\omega_{KE} ) = - \omega_{KE} $ in the complement of $D$. In order to solve  \ref{NEG KE} we use the Aubin-Yau continuity method. A novel feature is that the path we use $doesn't$ start at the reference metric $\omega$.

We start the continuity path with a metric whose Ricci potential is a smooth function rather than \(C^{\alpha}\), to obtain the initial metric we proceed as in the proof of Theorem \ref{Theorem1}. From now on we fix \(\alpha < (1/\beta)-1 \).  Consider the functional $\mathcal{F}: U \to C^{\alpha}$, where $U$ is a neighbourhood of $0$ in $C^{2, \alpha}$ and $\mathcal{F} (\phi) = \log (\omega_{\phi}^n / \omega^n) - \phi$. It is clear that $\mathcal{F} (0)=0$ and that the derivative at $0$ is given by $ D_0 \mathcal{F} (\phi) = \triangle_{g} \phi - \phi$. Integration by parts shows that $D_0 \mathcal{F}$ has no kernel, so that the Implicit Function Theorem together with Theorem \ref{LinearTheorem} imply that there is  $\epsilon>0$ such that for every $h \in C^{\alpha}$ with $\| h \|_{\alpha} < \epsilon$ there is $\phi \in C^{2, \alpha}$ with $\mathcal{F}(\phi) =h$. There is a function $f_0$, smooth in the complex coordinates, such that $\| f - f_0 \|_{\alpha} < \epsilon$. We let $h = f - f_0$ and take $\phi \in C^{2, \alpha}$ with $\mathcal{F} (\phi) = h$, so that $\omega_{\phi}$ satisfies
$ \omega_{\phi}^n = e^{h+\phi} \omega^n $; hence $\mbox{Ric}(\omega_{\phi}) = - \omega_{\phi} + i \partial \overline{\partial} f_0 $.

Set \( \omega_0 = \omega_{\phi} \). To solve equation \ref{NEG KE} it is enough to find $u_1 \in C^{2, \alpha}$ such that $ (\omega_0 + i \partial \overline{\partial} u_1)^n = e^{f_0 + u_1} \omega_0^n$; 
so then $u = \phi + u_1$ is the solution of \ref{NEG KE}. We use the path

\begin{equation} \label{CPath}
(\omega_0 + i\partial \overline{\partial} u_t)^n = e^{tf_0 + u_t} \omega_0^n 
\end{equation}
and consider the set
$$ T = \{ t \in [0, 1] \hspace{2mm} \mbox{such that there is} \hspace{2mm} u_t \in C^{2, \alpha} \hspace{2mm} \mbox{solving \ref{CPath}} \} . $$
We start at  $t=0$ with $u_0 =0$. The goal is to show that $T$ is open and closed.

Theorem \ref{LinearTheorem} implies that $T$ is open. The fact that $T$ is closed, and hence Theorem \ref{Theorem2}, follows from the following a priori estimate: 

\begin{proposition} \label{AprioriEstimate}
	There is a constant $C$, independent of $t \in T$, such that $\|u_t \|_{2, \alpha} \leq C$.
\end{proposition}

 The proof of Proposition \ref{AprioriEstimate} is divided into three steps:

 \emph{Step 1. $C^0$-estimate} This is an application of the maximum principle. If $u_t$ attains its maximum at $ p \in X \setminus D$ then \ref{CPath} implies that $t f_0 (p) + u_t (p) \leq 0$ , so that $ \sup u_t \leq  \max \{- \inf f_0, 0 \}$. If the maximum is attained at $p \in D$ then one considers $ \tilde{u}_t = u_t + \delta |s|_h^{\epsilon}$ for a suitable choice of $\delta$ and $\epsilon$ positive and small. The function $\tilde{u}_t$ attains its maximum outside $D$, one gets a uniform upper bound on the supreme of $ \tilde{u}_t$ -see \cite{Jeffress}- which indeed implies a uniform upper bound on $\sup u_t$. Similarly one gets a uniform lower bound on $ \inf u_t$. As a result $\| u_t \|_0 \leq C$.
	
 \emph{Step 2. $C^2$-estimate} 
 Write $\omega_t = \omega_0 + i \partial \overline{\partial}u_t$, then \ref{CPath} implies that $\mbox{Ric}(\omega_t) = - \omega_t + (1-t) i \partial \overline{\partial} f_0$.
 Since $f_0$ is smooth, there is a constant $C_2>0$ such that $i \partial \overline{\partial} f_0 \geq - C_2 \omega$. 
 Set $C_1 =1$ so that $\mbox{Ric}(\omega_t) \geq -C_1 \omega_t - C_2 \omega$. On the other hand, the reference metric $\omega$ has bisectional curvature bounded above, so there is $C_3 >0$ such that $\mbox{Bisec}(\omega) \leq C_3$.
 
 Write  $\tilde{u}_t = \phi + u_t$, so that $\omega_t = \omega + i \partial \overline{\partial} \tilde{u}_t$. Let $ A = C_2 + 2C_3 +1$; the Chern-Lu inequality -see \cite{JMR}- tells us that
	\begin{equation} \label{CHERN LU}
	\triangle_{\omega_t} ( \mbox{tr}_{\omega_t} \omega - A \tilde{u}_t) \geq -C_1 - An + \mbox{tr}_{\omega_t}\omega. 
	\end{equation}
We already have a uniform bound on $\| \tilde{u}_t \|_0$. We use \ref{CHERN LU} and the maximum principle (as in the previous step), together with the estimate on $\| \tilde{u}_t \|_0$, to get the uniform bound $ \mbox{tr}_{\omega_t} \omega \leq C$. This bound together with \ref{CPath}  imply that $ C^{-1} \omega \leq \omega_t \leq C \omega $.
	
\emph{Step 3. $C^{2, \alpha}$-estimate} This is a local result. We appeal to the `interior Schauder estimates for the complex Monge-Ampere operator'. In the case that $\beta =1$ (no cone singularities) there is a large literature on this topic; we mention, among others,  the work of Caffarelli and Safanov for the real Monge-Ampere operator. More recently, Chen-Wang -\cite{ChenWang}- gave a new proof of these estimates by means of a `blow-up' argument, similar in spirit to Leon Simon's proof of the Schauder estimates for the Laplace operator \cite{SimonSchauder}. This technique works in the setting of metrics with cone singularities. Our previous $C^2$ estimate together with Theorem 1.7 in \cite{ChenWang} gives us that $\| u \|_{2, \alpha} \leq C$.  Alternatively; we can refer to Evans-Krillov theory and its analogue for metrics with cone singularities, see \cite{JMR}.

\bibliographystyle{plain}
\bibliography{References}

\begin{thebibliography}{10}

\bibitem{Arezzocurvature}
Alberto Arezzo, Claudio; Della~Vedova and Gabriele La~Nave.
\newblock On the curvature of conic {K}ahler-{E}instein metrics.
\newblock {\em preprint}.

\bibitem{brendle}
Simon Brendle.
\newblock Ricci flat {K}\"ahler metrics with edge singularities.
\newblock {\em Int. Math. Res. Not. IMRN}, (24):5727--5766, 2013.

\bibitem{ChenWang}
Xiuxiong Chen and Yuanqi Wang.
\newblock $c^{2, \alpha}$-estimate for {M}onge-{A}mp{\`e}re equations with
  {H}{\"o}lder-continuous right hand side.
\newblock {\em Annals of Global Analysis and Geometry}, 49(2):195--204, 2016.

\bibitem{martin}
Martin de~Borbon.
\newblock {\em Asymptotically conical {R}icci-flat {K}\"ahler metrics with cone
  singularities}.
\newblock PhD thesis, Imperial College, London, UK, 2015.

\bibitem{donaldson1}
S.~K. Donaldson.
\newblock K\"ahler metrics with cone singularities along a divisor.
\newblock In {\em Essays in mathematics and its applications}, pages 49--79.
  Springer, Heidelberg, 2012.

\bibitem{GriffithsHarris}
Phillip Griffiths and Joseph Harris.
\newblock {\em Principles of algebraic geometry}.
\newblock John Wiley \& Sons, 2014.

\bibitem{JMR}
Thalia Jeffres, Rafe Mazzeo, and Yanir~A. Rubinstein.
\newblock K\"ahler-{E}instein metrics with edge singularities.
\newblock {\em Ann. of Math. (2)}, 183(1):95--176, 2016.

\bibitem{Jeffress}
Thalia~D Jeffres.
\newblock Uniqueness of {K}{\"a}hler-{E}instein cone metrics.
\newblock {\em Publicacions matematiques}, pages 437--448, 2000.

\bibitem{Rubinstein}
Yanir~A Rubinstein.
\newblock Smooth and singular {K}{\"a}hler-{E}instein metrics.
\newblock {\em Geometric and Spectral Analysis (P. Albin et al., eds.),
  Contemp. Math}, 630:45--138, 2014.

\bibitem{SimonSchauder}
Leon Simon.
\newblock Schauder estimates by scaling.
\newblock {\em Calculus of Variations and Partial Differential Equations},
  5(5):391--407, 1997.

\end{thebibliography}

\end{document}